\documentclass[12pt]{amsart}

\usepackage[all]{xy}
\usepackage{enumerate}

\usepackage[bookmarksnumbered=true]{hyperref} 
\usepackage[normalem]{ulem}
\hypersetup{
     colorlinks = true,
     linkcolor = blue,
     anchorcolor = blue,
     citecolor = blue,
     filecolor = blue,
     urlcolor = blue
}
\usepackage{url}
\usepackage{graphicx}%
\usepackage{multirow}%
\usepackage{amsmath,amssymb,amsfonts}%
\usepackage{amsthm}%
\usepackage{mathrsfs}%
\usepackage[title]{appendix}%
\usepackage{xcolor}%
\usepackage{textcomp}%
\usepackage{manyfoot}%
\usepackage{booktabs}%
\usepackage{algorithm}%
\usepackage{algorithmicx}%
\usepackage{algpseudocode}%
\usepackage{listings}%

\usepackage{enumitem}
\usepackage[normalem]{ulem} 

\newtheorem{theorem}{Theorem}[section]
\newtheorem{corollary}[theorem]{Corollary}
\newtheorem{proposition}[theorem]{Proposition}
\newtheorem{lemma}[theorem]{Lemma}

\theoremstyle{definition}

\newtheorem{example}[theorem]{Example}
\theoremstyle{remark}
\newtheorem{remark}[theorem]{Remark}
\numberwithin{equation}{section}

\title{Counting finite $O$-sequences of a given multiplicity}

\author[F.~Cioffi]{Francesca Cioffi}
\address{Dip.~di Matematica e Appl. \\ Universit\`a degli Studi di Napoli Federico II\\ Via Cintia \\ 80126 Napoli \\ Italy.}
\email{\href{mailto:cioffifr@unina.it}{cioffifr@unina.it}}

\author[M.~Guida]{Margherita Guida}
\address{Dip.~di Matematica e Appl. \\ Universit\`a degli Studi di Napoli Federico II\\ Via Cintia \\ 80126 Napoli \\ Italy.}
\email{\href{mailto:maguida@unina.it}{maguida@unina.it}}

\subjclass[2020]{05A15, 05A16, 11B83, 13P10}
\keywords{O-sequence, lex-segment ideal, decomposition of order ideals}

\begin{document}

\begin{abstract}We study the number $O_d$ of finite $O$-sequences of a given multiplicity $d$, with particular attention to the computation of $O_d$. We show that the sequence $(O_d)_d$ is sub-Fibonacci, and that if the sequence $(O_d / O_{d-1})_d$ converges, its limit is bounded above by the golden ratio. This analysis also produces an elementary method for computing $O_d$. In addition, we derive an iterative formula for $O_d$ by exploiting a decomposition of lex-segment ideals introduced by S.~Linusson in a previous work.
\end{abstract}

\maketitle

\section*{Introduction}

Let $O_d$ denote the number of all finite $O$-sequences of a given multiplicity $d$. Finite $O$-sequences correspond to the $h$-vectors of Cohen–Macaulay quotient rings of polynomial graded algebras with respect to the standard grading. As such, the sequence~$(O_d)_d$ has drawn the interest of several authors.

We recall that in~\cite{RR}, L.~Roberts posed the question if the sequence $(O_d/O_{d-1})_d$ converges to a number strictly greater than $1$ as $d$ increases. In~\cite{SL}, an interesting iterative formula for computing the number of finite $O$-sequences under specific conditions, other than the multiplicity, is presented (see \cite[Formula (3)]{SL}), along with some asymptotic estimates. In~\cite[Theorem 2.4]{ES}, it is shown that $(O_d)_d$ is bounded above by the Fibonacci sequence, and a lower bound is provided in terms of integer partitions. In~\cite{SZ}, both upper and lower bounds for $O_d$ are significantly improved. A natural combinatorial description of finite $O$-sequences via suitably defined connected graphs is given in~\cite{CLM}. Related problems, such as the enumeration of stable and strongly stable ideals under certain constraints, are studied, for example, in~\cite{Ceria} (see also the references therein).

In this paper, we show that the sequence $(O_d)_d$ is sub-Fibonacci, according to an extension of the definition given by P.~C. Fishburn and F.~S. Roberts in \cite{FR}, and observe that if the sequence $(O_d/O_{d-1})_d$ converges, then its limit must be a real number $b$ such that $1 \leq b \leq \frac{1+\sqrt{5}}{2}$, i.e.~the golden ratio, which is also the limit of the ratio of the Fibonacci sequence. This analysis also produces an elementary method for computing~$O_d$.  Moreover, by revisiting the insightful decomposition of the sous-escalier of lex-segment ideals introduced by S. Linusson in~\cite{SL}, originally used to obtain the already quoted  \cite[Formula (3)]{SL}, we derive an iterative formula for $O_d$.

\section{Preliminaries}
\label{sec:prelim}

Let $R:=K[x_1,\dots,x_p]$ be the polynomial ring over a field $K$ with the variables ordered as $x_1<\dots<x_p$. A term of $R$ is a power product $x^\alpha=x_1^{\alpha_1}\dots x_p^{\alpha_p}$, with $\alpha_i\in \mathbb Z_{\geq 0}$ for every $i\in\{1,\dots,p\}$. A monomial ideal $J$ of $R$ is an ideal which admits a set of generators made of terms. 

An {\em order ideal} is a set of terms closed under division. Equivalently, it is the {\em sous-escalier} $\mathcal N(J)$ of a monomial ideal $J$, i.e.~the set of the terms outside~$J$. 

A monomial ideal $J$ is a {\em lex-segment ideal} if, for every degree $t$ and for every term $\tau \in \mathcal N(J)$ of degree $t$, if $\sigma$ is a term of degree $t$ lower than $\tau$ with respect to the lexicographic term order, then $\sigma$ belongs to~$\mathcal N(J)$.  

For a homogeneous ideal $I\subset R$, the Hilbert function of the $K$-graded algebra $R/I$ is the function $H_{R/I}: t\in \mathbb Z_{\geq 0}\to \dim(R_t) - \dim(I_t) \in \mathbb Z_{\geq 0}$. Equivalently, we may say that $H_{R/I}$ is the Hilbert function of $I$.

A numerical function $H=(a_0,a_1,\dots,a_t,\dots)$ which is the Hilbert function of a $K$-algebra $R/I$ is also called {\em an admissible function or an $O$-sequence}.

Recall that, given two positive integers $a$ and $t$, the {\em binomial expansion of $a$ in base $t$} is the unique writing
$$a  =  \binom{k(t)}{t} + \binom{k(t-1)}{t-1} + \dots + \binom{k(j)}{j} =: a_t,
$$
where $k(t)> k(t-1)>\dots > k(j)\geq j\geq 1$, and with the convention that a binomial coefficient $\binom{n}{m}$ is null whenever either $n<m$ or $m<0$ 
and $\binom{n}{0}=1$, for all $n\geq 0$. Let 
$$
a_t^{\langle t\rangle}:=\binom{k(t)+1}{t+1} + \binom{k(t-1)+1}{t} + \dots + \binom{k(j)+1}{j+1}.
$$
It is well-known that, if $a_t$ is the value assumed by a Hilbert function at a degree~$t$, then $a_t^{\langle t\rangle}$ is the maximum value that this Hilbert function can assume at degree~$t +1$, i.e.
\begin{equation}\label{eq:condition}
a_{t+1} \leq  a_t^{\langle t\rangle}.
\end{equation} 
As independently studied in \cite{Ma} and \cite{H66}, this value depends on the growth of the lex-segment ideals and characterize the Hilbert functions. Indeed, for every Hilbert function $H=(a_0,a_1,\dots)$ there is a unique lex-segment ideal $L\subseteq K[x_1,\dots,x_{a_1}]$ such that $H$ is the Hilbert function of $K[x_1,\dots,x_{a_1}]/L$. Hence, there is a bijective correspondence between the set of Hilbert functions and the set of lex-segment ideals.

Here, we will focus on finite $O$-sequences $H=(a_0,a_1,\dots,a_s)$, $a_s\not=0$, which are the Hilbert functions of the Artinian quotients $R/I$. The integer $s$ is called the {\em socle degree} of $H$ and $\sum a_i$ is the {\em multiplicity or length} of $H$. 

From now, for every positive integer $d$, the symbol $O_d$ denotes the number of all finite $O$-sequences of multiplicity $d$.
The symbol $A_d$ denotes the number of all finite $O$-sequences of multiplicity $d$ such that the last non-zero value is strictly greater than~$1$.

\section{Relations between $O_d$ and the Fibonacci sequence}

A non-decreasing integer sequence $(x_k)_k$, for which $x_1=x_2=1$, is {\em sub-Fibonacci} if $x_k\leq x_{k-1}+x_{k-2}$ for all $k\geq 3$ (see \cite[page 262]{FR} for finite sequences).

\begin{lemma}\label{lemma:prima}
$O_d=O_{d-1}+A_d$.
\end{lemma}

\begin{proof}
We observe that, for every $O$-sequence $(a_0,a_1,\dots,a_s)$ of multiplicity $d-1$, the sequence $(a_0,a_1,\dots,a_s,1)$ is one of the $O$-sequences of multiplicity $d$ with last non-null value equal to $1$ and, if $1\leq a_s < a_{s-1}^{\langle s-1\rangle}$, then $(a_0,a_1,\dots,a_s+1)$ is one of the $A_d$ $O$-sequences of multiplicity~$d$ with last non-null value strictly greater than $1$. The vice versa also holds. 
\end{proof}

\begin{lemma}\label{lemma: Ad}
$A_{d-2}\leq A_d \leq A_{d-1}+A_{d-2}=O_{d-1}-O_{d-3}$, for every $d\geq 4$, where the second inequality is strict for every $d\geq 7$.
\end{lemma}

\begin{proof}
The first inequality follows from the observation that, given an $O$-sequence $(a_0,a_1,\dots,a_s)$ of multiplicity $d-2$ with $a_s>1$, then $(a_0,a_1,\dots,a_s,2)$ is an $O$-sequences of multiplicity $d$ with last non-zero value greater than $1$.

For the second inequality, first observe that $ A_{d-1}+A_{d-2}=O_{d-1}-O_{d-3}$ by Lemma~\ref{lemma:prima}. Then, considering the construction described in the proof of Lemma \ref{lemma:prima}, note that there are exactly $O_{d-3}$ $O$-sequences of multiplicity $d-1$ of type $(a_0,a_1,\dots,1,1)$ that cannot give rise to $O$-sequences of multiplicity $d$ with last non-zero value strictly greater than $1$. Hence, $A_d\leq O_{d-1}-O_{d-3}$. Moreover, for every $d\geq 7$, there is also at least the $O$-sequence of multiplicity $d-1$ which ends with a number $>1$, of type $(1,2,2,2\dots)$ or $(1,3,2,\dots)$, which cannot give rise to $O$-sequences of multiplicity $d$ with last non-zero value strictly greater than $1$.
\end{proof}

\begin{proposition}\label{prop: sub-fibonacci}
$O_d \leq O_{d-1}+O_{d-2}$, for every $d\geq 3$.
\end{proposition}

\begin{proof}
First observe that $O_3=2=1+1=O_1+O_2$, $O_4=3=2+1=O_3+O_2$, $O_5=5=3+2=O_4+O_3$ and $O_6=8=5+3=O_5+O_4$. For every $d\geq 7$, applying repeatedly Lemma~\ref{lemma:prima}, we obtain $O_d=\sum_{j=1}^d A_j +1=\sum_{j=3}^d A_j+1$, being $A_1=A_2=0$. Then, by Lemma \ref{lemma: Ad} 

$A_d+A_{d-1}+1 \leq O_{d-1}-O_{d-3} + O_{d-2}-O_{d-4}$

$A_d+A_{d-1}+A_{d-2}+1 \leq O_{d-1}-O_{d-3} + O_{d-2}-O_{d-4} + O_{d-3} - O_{d-5}$

$\vdots$

$O_d=\sum_{j=1}^d A_j +1 \leq O_{d-1}+O_{d-2}$.
\end{proof}

\begin{corollary}\label{cor: sub-fibonacci}
The sequence $(O_d)_d$ is sub-Fibonacci.
\end{corollary}

\begin{proof}
Observe that $O_1=O_2=1$ and $(O_d)_d$ is not decreasing by Lemma~\ref{lemma:prima}. Then the statement follows from Proposition \ref{prop: sub-fibonacci}. 
\end{proof}

Now we highlight some features of the sequence $(O_d/O_{d-1})_d$, applying classical methods already used to study the growth rate of the Fibonacci sequence. 

First we observe that Lemma~\ref{lemma:prima} straightforwardly implies that the sequence $(O_d/O_{d-1})_d$ is decreasing if, and only if, $(A_d/O_{d-1})_d$ is decreasing.

\begin{proposition}\label{prop: rate}
\
\begin{enumerate}
\item \label{limiti} $1< O_d/O_{d-1}\leq 2$ and $O_d\leq 2^{d-2}$, for every $d\geq 3$.

\item \label{equivalenza} $(O_d/O_{d-1})_d$ converges to the real number $b$ if, and only if, the sequence $(A_d/O_{d-1})_d$ converges to the real number $\ell=b-1$.

\item \label{sezione aurea} If $(O_d/O_{d-1})_d$ converges to a real number $b$, then $b\leq \frac{1+\sqrt{5}}{2}$ and, equivalently, $\ell \leq (\frac{1+\sqrt{5}}{2})^{-1}$, where $\ell$ is the limit of $(A_d/O_{d-1})_d$.

\item \label{quoziente} $\displaystyle{\frac{A_d}{O_{d-1}} \leq \frac{A_{d-1}}{O_{d-2}}+\frac{A_{d-2}}{O_{d-3}}}$ and $\displaystyle{\frac{O_d}{O_{d-1}} \leq \frac{O_{d-1}}{O_{d-2}}+\frac{O_{d-2}}{O_{d-3}}}$.
\end{enumerate}
\end{proposition}

\begin{proof}
From Lemmas \ref{lemma:prima} and \ref{lemma: Ad}, we obtain   
$O_d/O_{d-1}=\frac{O_{d-1}+A_d}{O_{d-1}}=1+\frac{A_d}{O_{d-1}} < 2$, so item~\eqref{limiti} holds because $A_d\leq O_{d-1}$ and $O_3=2$. 
Item~\eqref{equivalenza} follows from item  \eqref{limiti} and Lemma~\ref{lemma:prima}. 

For item \eqref{sezione aurea}, $O_d/O_{d-1} \leq \frac{O_{d-1}+O_{d-2}}{O_{d-1}}=1+\frac{O_{d-2}}{O_{d-1}}$ thanks to Proposition~\ref{prop: sub-fibonacci}. By the hypothesis $(O_{d-1}/O_{d})_d$ converges to $1/b$, hence $b \leq 1+1/b$ and so $b^2-b-1\leq 0$, that is $b\leq \frac{1+\sqrt{5}}{2}$. Finally, $\ell \leq (\frac{1+\sqrt{5}}{2})^{-1}$ follows because $b=1+\ell$ by item \eqref{equivalenza} and Lemma \ref{lemma:prima}.

For item \eqref{quoziente}, it is enough to observe that by Lemma \ref{lemma: Ad}
$$\frac{A_d}{O_{d-1}} \leq \frac{A_{d-1}+A_{d-2}}{O_{d-1}}=\frac{A_{d-1}}{O_{d-1}}+\frac{A_{d-2}}{O_{d-1}}\leq \frac{A_{d-1}}{O_{d-2}}+\frac{A_{d-2}}{O_{d-3}}$$ 
The second inequality analogously follows from Proposition~\ref{prop: sub-fibonacci}.
\end{proof}

\begin{remark}
Proposition \ref{prop: rate}\eqref{limiti} guarantees that the sequence $(O_d/O_{d-1})_d$ has a convergent subsequence, by the Bolzano-Weiestrass Theorem. 

However, we can refine Proposition \ref{prop: rate}\eqref{limiti}(1) in the following way. If at a step $t$ the sequence $(O_d/O_{d-1})_{d}$ assumes a value $>\frac{1+\sqrt{5}}{2}$, then at step $t+1$ it must assume a value $< \frac{1+\sqrt{5}}{2}$. Indeed, if $O_t/O_{t-1}>\frac{1+\sqrt{5}}{2}$ then $O_{t+1}/O_t \leq 1+O_{t-1}/O_{t}< 1 + \frac{2}{1+\sqrt{5}}=\frac{3+\sqrt{5}}{1+\sqrt{5}}=\frac{1+\sqrt{5}}{2}$ by Proposition~\ref{prop: sub-fibonacci}, according to the result of Proposition \ref{prop: rate}\eqref{limiti}(3) in the hypothesis that the sequence converges. More precisely, we obtain 
\begin{equation}\label{eq:bound} 
O_d\leq \Bigl(\frac{5}{3}\cdot \frac{1+\sqrt{5}}{2}\Bigr)^{\lfloor\frac{d-2}{2}\rfloor}.
\end{equation}
Indeed, from Proposition \ref{prop: sub-fibonacci} the inequalities $O_d \leq O_{d-1}+O_{d-2}\leq O_{d-2}+O_{d-3} + O_{d-2} < 3 O_{d-2}$ follow, because $O_{d-3}<O_{d-2}$. Hence, by Lemma \ref{lemma: Ad} we obtain $A_d\leq O_{d-1}-O_{d-3} \leq O_{d-1}-\frac{1}{3}O_{d-1}=\frac{2}{3} O_{d-1}$ and then
\begin{equation}\label{eq:new bound} 
\frac{O_d}{O_{d-1}}=1+\frac{A_d}{O_{d-1}}\leq 1+\frac{2}{3}=\frac{5}{3}.
\end{equation}
We conclude applying repeatedly this inequality and taking into account the previous observations and that $O_1=O_2=1$.
\end{remark}

\section{An elementary computation of $O$-sequences}

The proofs of Lemmas \ref{lemma:prima} and \ref{lemma: Ad} suggest a very easy way to construct the set of all the finite $O$-sequences of a given multiplicity $d$ and, consequently, to compute~$O_d$. We collect the details of this construction in Algorithm~\ref{algorithm}. An implementation in CoCoA~5 (see \cite{CoCoA}) of this algorithm is available at \url{https://www.dma.unina.it/~cioffi/MaterialeOsequences/OSequences.CoCoA5}.

\begin{algorithm}[!ht]
	\caption{\label{algorithm} Construction of all the finite $O$-sequences of multiplicity $d$ and consequent computation of~$O_d$ and $A_d$, for every $1\leq d\leq D$, with $D\geq 4$ a given positive integer}
	\begin{algorithmic}[1]
		\State $O$Sequences$\left(D\right)$
		\Require a positive integer $D$
		\Ensure the value $O_d$, for every $1\leq d\leq D$ 
		\State $O:=[1,1,2,3]$
		\State $A:=[0,0,1,1]$
		\State $B:=[[[1,2]],[1,3],[]]$
		\State $h_2:=1$, $h_1:=2$, $h=3$
		\For{$d=5$ to $D$}
		\State $B[h]$ vector of the $O$-sequences obtained from:  those in $\mathcal B[h_2]$ {\em attaching} a $2$ to the end, and those in $\mathcal B[h_1]$ {\em increasing by $1$} the last non-null value, if possible
		\State Attach the cardinality of $B[h]$ to the end of the list $A$
		\State Attach the sum $A[d]+O[d-1]$ to the end of the list $O$
		\State $a:=h_2$, $h_2:=h_1$, $h_1:=h$, $h:=a$ 
		\EndFor
		\State return $O$ 
	\end{algorithmic}
\end{algorithm}

\begin{proposition}
Given a positive integer $D\geq 4$, Algorithm~\ref{algorithm} returns the values~$O_d$, for every multiplicity $1\leq d\leq D$. 
\end{proposition}

\begin{proof}
The algorithm $O$Sequences$\left(D\right)$ deals with a finite number of objects that are described by a finite number of data each, hence it is clear that the procedure terminates when $d=D$. For the correctness, we now analyse the command lines.

Lines 2-5 contain instruction for the assignment of the lists $O$ of the integers $O_d$, $A$ of the integers $A_d$ and $B$ which contains three lists: $B[h_2]$ stores the $O$-sequences of multiplicity $d-2$ with last value greater than $1$, $B[h_1]$ stores the analogous $O$-sequences of multiplicity $d-1$, and $B[h]$ contains the analogous $O$-sequences of multiplicity $d$. The algorithm avoids to store the $O$-sequences with last non-null value equal to $1$.

Into a structure of type \lq\lq for", the instruction on line 7 updates $B[h]$, taking into account the proofs of Lemmas \ref{lemma:prima} and \ref{lemma: Ad}, which show that the $O$-sequences of multiplicity $d$ with last value greater than $1$ can be obtained from the analogous $O$-sequences of multiplicity $d-2$ by attaching a \lq\lq $2$" to the end of each $O$-sequence, and from the analogous $O$-sequences of multiplicity $d-1$, increasing the last non-null value by $1$ if possible, according to the properties of $O$-sequences (see formula~\eqref{eq:condition}). 

On lines 8-9 the value of $A_d$ is assigned and stored in the list $A$ and that of $O_d$ in the list $O$, respectively. Then, the values of the indexes $h_2, h_1, h$ are updated  for a possible new step corresponding to multiplicity $d+1$.
\end{proof}

\begin{remark}
Although the rapid growth of the number of $O$-sequences of multiplicity~$d$ poses challenges to the efficiency of Algorithm~\ref{algorithm}, we have computed the values of~$O_d$ up to multiplicity $d = 60$. This data (available at \url{https://oeis.org/A232476} for $d = 1, \dots, 20$ and in Table~\ref{table} for $d = 21, \dots, 60$) shows that the sequence $(O_d / O_{d-1})_d$ is in fact decreasing, at least for integers $d \in [6, 60] \cap \mathbb{N}$, as expected.
\end{remark}

\section{An iterative formula for $O_d$}
\label{sec:iterative}

For every integer $p>0$, $n\geq 0$, $k\geq 0$, $d>0$, let  $M(p,n,k,d)$ denote  the set of all sous-escaliers of Artinian lex-segment ideals in at most $p$ variables, corresponding to finite $O$-sequences $(a_0,a_1,\dots,a_s)$ of multiplicity~$d$, satisfying the following conditions:

$\bullet$ the socle degree $s$ is at most $n$,

$\bullet$ $a_i=\binom{p-1+i}{i}$ for all $0\leq i\leq k$, and 

$\bullet$ $a_i<\binom{p-1+i}{i}$ for all $i>k$. 
\vskip 1mm
\noindent We denote by $O(p,n,k,d)$ the cardinality of $M(p,n,k,d)$. 

Referring to \cite{SL}, let $M\subseteq K[x_1,\dots,x_p]$ be an order ideal that is the sous-escalier of a lex-segment ideal $J$. Then, we consider the following subsets of~$M$ 
\begin{equation}\label{eq: M}
M_1:=\{ \tau \in M \ : \ x_{p} \not\vert \tau\}, \qquad M_2:=\{\frac{\sigma}{x_{p}} : \sigma\in M\setminus M_1\}.
\end{equation} 

\begin{proposition}\label{prop: decomposizione}
With the notation above, for every $p\geq 2$, $M$ belongs to $M(p,n,k,d)$ if, and only if, $M_1$ belongs to $M(p-1,n,i,d-j)$ and $M_2$ belongs to $M(p,i-1,k-1,j)$.
\end{proposition}

\begin{proof}
By the definition of lex-segment ideal, $M$ is the sous-escalier of a lex-segment ideal if and only if both $M_1$ and $M_2$ are   sous-escaliers of lex-segment ideals. 
Then, we recall that the proof of \cite[Theorem 2.1]{SL} guarantees that $M_1$ is an order ideal in a number of variables $\leq p-1$ with $O$-sequence having socle degree $\leq n$ and attaining the maximal value up to a suitable integer $i$ such that $k\leq i \leq n$, by construction. Analogously, $M_2$ is the order ideal in a number of variables $\leq p$, with an $O$-sequence of socle degree $\leq i-1$ and attaining the maximal value up to $k-1$, where $i$ is the same integer involved in the description of~$M_1$. 

Moreover, observing that the multiplicity of the $O$-sequence of $M$ is the cardinality of~ $M$, which is the sum of the cardinalities of $M_1$ and $M_2$, by construction, if $M \in M(p,n,k,d)$ then $M_1\in M(p-1,n,i,d-j)$ and $M_2 \in M(p,i-1,k-1,j)$, for a suitable integer $0<j<d$. 

The vice versa is obtained following the inverse constructive procedure.
\end{proof}

\begin{lemma}\label{lemma: base}
\
\begin{enumerate}
\item[(i)] $O_d=O(d,d-1,0,d)$.
\item[(ii)] $O(1,n,k,d)= \left\{\begin{array}{cl}1, &\text{ if } k=d-1 \text{ and } n\geq d-1\\ 0, &\text{ otherwise}\end{array}\right.$.
\item[(iii)] $O(p,n,0,d)=\sum_{k=0}^{d-1} O(p-1,n,k,d)$, for every $p\geq 2$.
\end{enumerate}
\end{lemma}

\begin{proof}
The statement follows from the definition of $O(p,n,k,d)$. 
\end{proof}

\begin{theorem}\label{th: iterative formula}
For every integer $p>1$, $n\geq 0$, $k> 0$, $d>0$, we can compute $O(p,n,k,d)$ by the following formula:
\begin{equation}\label{eq:formula k}
O(p,n,k,d)=\sum_{j=1}^{d-1} \sum_{i=k}^n O(p-1,n,i,d-j)\cdot O(p,i-1,k-1,j).
\end{equation}
\end{theorem}

\begin{proof}
Considering items $(ii)$ and $(iii)$ of Lemma \ref{lemma: base} as base of induction, the iterative formula~\eqref{eq:formula k} follows from the bijection described in Proposition \ref{prop: decomposizione}.
\end{proof}

\begin{corollary}
Lemma \ref{lemma: base} and Theorem \ref{th: iterative formula} provide an iterative formula to compute~$O_d$, for every $d\geq 2$.
\end{corollary}

\begin{example}
Beyond computing the values of $O_d$, formula~\eqref{eq:formula k} also enables the computation of the cardinality of other sets of monomial ideals. For instance, when $p = 2$, the set of lex-segment ideals of a given multiplicity $d$ coincides with both the set of strongly stable ideals and the set of stable ideals of the same multiplicity (e.g., see~\cite{Ceria}). Therefore, the cardinality of this set is given by $O(3, d - 1, 0, d)$.
\end{example}

An implementation in CoCoA~5 of formula \eqref{eq:formula k} is available at 
 
\hskip 7.5mm\url{https://www.dma.unina.it/~cioffi/MaterialeOsequences/IterativeFormulaOSequences.CoCoA5}.

\section*{Acknowledgements}
The first author is a member of GNSAGA (INdAM, Italy).

\begin{table}[]
\caption{Values of $O_d$ for $d\in\{21,\dots,60\}$.}\label{table}
\begin{tabular}{r| rrrrr}
\hline
$d$ & $21$ & $22$ & $23$ & $24$ & $25$ \\
$O_d$  & $1416$ &$1882$ &$2490$ &$3279$ &$4299$ \\
\hline\hline
$d$ & $26$ & $27$ & $28$ & $29$ & $30$\\
$O_d$ &$5612$ &$7297$ &$9451$ &$12195$ &$15683$  \\
\hline\hline
$d$ & $31$ & $32$ & $33$ & $34$ & $35$  \\
$O_d$ & $20099$ &$25674$ &$32696$ &$41514$ &$5255$ \\
\hline\hline
$d$ & $36$ & $37$ & $38$ & $39$ & $40$ \\
$O_d$ &$66361$ &$83561$ &$104951$ &$131491$ &$164347$ \\
\hline\hline
$d$ & $41$ & $42$ & $43$ & $44$ & $45$\\
$O_d$ & $204936$ &$254979$ & $316552$ & $392166$ & $484853$ \\
\hline\hline
$d$ &$46$ & $47$ & $48$ & $49$ & $50$\\
$O_d$ & $598255$ & $736759$ & $905635$ & $1111194$ & $1360997$\\

\hline\hline
$d$ & $51$ & $52$ & $53$ & $54$ & $55$ \\
$O_d$  &$1664090$ & $2031266$ & $2475404$& $3011853$& $3658861$ \\
\hline\hline
$d$ & $56$ & $57$ & $58$ & $59$ & $60$\\
$O_d$  &$4438118$& $5375378$& $6501163$& $7851624$& $9469536$ \\
\hline
\end{tabular}
\end{table} 


\providecommand{\bysame}{\leavevmode\hbox to3em{\hrulefill}\thinspace}
\providecommand{\MR}{\relax\ifhmode\unskip\space\fi MR }
\providecommand{\MRhref}[2]{%
  \href{http://www.ams.org/mathscinet-getitem?mr=#1}{#2}
}
\providecommand{\href}[2]{#2}

\end{document}